\documentclass[12pt]{article}

\makeatletter
\let\@fnsymbol\@arabic
\makeatother

\usepackage{latexsym,amssymb,amsthm,amsmath,mathtools,epsfig,color,epsf}
\usepackage{graphicx,tikz,pgfplots,pgfplotstable}
\usepackage{enumitem}

\newlist{compactenum}{enumerate}{4}
\setlist[compactenum,1]{nolistsep}

\usepgfplotslibrary{polar}
\usetikzlibrary{arrows.meta}

\makeatletter

\numberwithin{equation}{section}

\def\@noindentfalse{\global\let\if@noindent\iffalse}
\def\@noindenttrue {\global\let\if@noindent\iftrue}
\def\@aftertheorem{%
  \@noindenttrue
  \everypar{%
    \if@noindent%
      \@noindentfalse\clubpenalty\@M\setbox\z@\lastbox%
    \else%
      \clubpenalty \@clubpenalty\everypar{}%
    \fi}}

\theoremstyle{plain}
\newtheorem{theorem}{Theorem}[section]
\AfterEndEnvironment{theorem}{\@aftertheorem}

\AfterEndEnvironment{definition}{\@aftertheorem}
\newtheorem{lemma}[theorem]{Lemma}
\AfterEndEnvironment{lemma}{\@aftertheorem}
\newtheorem{corollary}[theorem]{Corollary}
\AfterEndEnvironment{corollary}{\@aftertheorem}

\theoremstyle{definition}
\newtheorem{remark}[theorem]{Remark}
\AfterEndEnvironment{remark}{\@aftertheorem}
\newtheorem{assumption}[theorem]{Assumption}
\AfterEndEnvironment{assumption}{\@aftertheorem}
\newtheorem{example}[theorem]{Example}
\AfterEndEnvironment{example}{\@aftertheorem}

\AfterEndEnvironment{proof}{\@aftertheorem}

\makeatletter

\def\note#1{\par\smallskip%
\noindent\kern-0.01\hsize%
\setlength\fboxrule{0pt}\fbox{\setlength\fboxrule{0.5pt}\fbox{%
\llap{$\boldsymbol\Longrightarrow$ }%
\vtop{\hsize=0.98\hsize\parindent=0cm\small\rm #1}%
\rlap{$\enskip\,\boldsymbol\Longleftarrow$}
}}%
}

\usepackage{delimset}

\def\given{\mskip 0.5mu plus 0.25mu\vert\mskip 0.5mu plus 0.15mu}
\newcounter{bracketlevel}%
\def\@bracketfactory#1#2#3#4#5#6{%
\expandafter\def\csname#1\endcsname##1{%
\global\advance\c@bracketlevel 1\relax%
\global\expandafter\let\csname @middummy\alph{bracketlevel}\endcsname\given%
\global\def\given{\mskip#5\csname#4\endcsname\vert\mskip#6}\csname#4l\endcsname#2##1\csname#4r\endcsname#3%
\global\expandafter\let\expandafter\given\csname @middummy\alph{bracketlevel}\endcsname%
\global\advance\c@bracketlevel -1\relax%
}%
}
\def\bracketfactory#1#2#3{%
\@bracketfactory{#1}{#2}{#3}{relax}{0.5mu plus 0.25mu}{0.5mu plus 0.15mu}
\@bracketfactory{b#1}{#2}{#3}{big}{1mu plus 0.25mu minus 0.25mu}{0.6mu plus 0.15mu minus 0.15mu}
\@bracketfactory{bb#1}{#2}{#3}{Big}{2.4mu plus 0.8mu minus 0.8mu}{1.8mu plus 0.6mu minus 0.6mu}
\@bracketfactory{bbb#1}{#2}{#3}{bigg}{3.2mu plus 1mu minus 1mu}{2.4mu plus 0.75mu minus 0.75mu}
\@bracketfactory{bbbb#1}{#2}{#3}{Bigg}{4mu plus 1mu minus 1mu}{3mu plus 0.75mu minus 0.75mu}
}
\bracketfactory{clc}{\lbrace}{\rbrace}
\bracketfactory{clr}{(}{)}
\bracketfactory{cls}{[}{]}
\bracketfactory{abs}{\lvert}{\rvert}
\bracketfactory{norm}{\Vert}{\Vert}
\bracketfactory{floor}{\lfloor}{\rfloor}
\bracketfactory{ceil}{\lceil}{\rceil}
\bracketfactory{angle}{\langle}{\rangle}

\newcounter{ctr}\loop\stepcounter{ctr}\edef\X{\@Alph\c@ctr}%
\expandafter\edef\csname s\X\endcsname{\noexpand\mathscr{\X}}
\expandafter\edef\csname c\X\endcsname{\noexpand\mathcal{\X}}
\expandafter\edef\csname b\X\endcsname{\noexpand\boldsymbol{\X}}
\expandafter\edef\csname I\X\endcsname{\noexpand\mathbb{\X}}
\ifnum\thectr<26\repeat

\let\@IE\IE\let\IE\undefined
\newcommand{\IE}{\mathop{{}\@IE}\mathopen{}}
\let\@IP\IP\let\IP\undefined
\newcommand{\IP}{\mathop{{}\@IP}}

\newcommand{\bigo}{\mathop{{}\mathrm{O}}\mathopen{}}

\newcommand{\E}{\mathbb{E}}
\newcommand{\R}{\mathbb{R}}

\newcommand{\eps}{\varepsilon}

\newcommand{\mel}{\MoveEqLeft}

\renewcommand{\phi}{\varphi}
\renewcommand{\epsilon}{\varepsilon}

\let\original@left\left
\let\original@right\right
\renewcommand{\left}{\mathopen{}\mathclose\bgroup\original@left}
\renewcommand{\right}{\aftergroup\egroup\original@right}

\def\^#1{\relax\ifmmode {\mathaccent"705E #1} \else {\accent94 #1} \fi}
\def\˘#1{\relax\ifmmode {\mathaccent"7014 #1} \else {\check{#1}} \fi}
\def\~#1{\relax\ifmmode {\mathaccent"707E #1} \else {\accent"7E #1} \fi}
\def\*#1{\relax#1^\ast}
\edef\-#1{\relax\noexpand\ifmmode {\noexpand\bar{#1}} \noexpand\else \-#1\noexpand\fi}
\def\>#1{\vec{#1}}
\def\.#1{\dot{#1}}

\def\atop{\@@atop}
\def\%#1{\mathcal{#1}}

\renewcommand{\P}{\mathbb{P}}
\renewcommand{\E}{\mathbb{E}}
\renewcommand{\R}{\mathbb{R}}
\newcommand{\N}{\mathbb{N}}

\newcommand{\h}{1}
\newcommand{\s}{0}

\newcommand{\e}{\mathrm{e}}
\renewcommand{\d}{\mathrm{d}}
\def\cL{\mathcal{L}}
\def\cM{\mathcal{M}}

\def\cE{\mathcal{E}}

\makeatother

\begin{document}


\title{The Moran model with random resampling rates}


\author{Siva Athreya, \quad  Frank den Hollander  and Adrian R\"ollin}

\newcommand{\Addresses}{{
  \bigskip
  \footnotesize

  Siva Athreya, \textsc{ International Centre for Theoretical Sciences, Survey No.\ 151, Shivakote, Hesaraghatta Hobli, Bengaluru 560089, India and Indian Statistical Institute, 8th Mile Mysore Road, Bengaluru 560059, India.}  \par \nopagebreak \texttt{athreya@icts.res.in},

  \medskip
  Frank den Hollander, \textsc{Mathematical Institute, Leiden University, Niels Bohrweg 1, 2333 CA  Leiden, The Netherlands. }
  \par \nopagebreak \texttt{ denholla@math.leidenuniv.nl}

  \medskip
  
  Adrian R\"ollin, \textsc{Department of Statistics and Data Science, National University of Singapore, 6 Science Drive 2, Singapore 117546. }
 \par \nopagebreak \texttt{adrian.roellin@nus.edu.sg}
  
}}

\date{}

\maketitle

\begin{abstract}
In this paper we consider the two-type Moran model with $N$ individuals. Each individual is assigned a resampling rate, drawn independently from a probability distribution $\P$ on $\R_+$, and a type, either $\h$ or $\s$. Each individual resamples its type at its assigned rate, by adopting the type of an individual drawn uniformly at random. Let $Y^N(t)$ denote the empirical distribution of the resampling rates of the individuals with type $\h$ at time $Nt$. We show that if $\P$ has countable support and satisfies certain tail and moment conditions, then in the limit as $N\to\infty$ the process $(Y^N(t))_{t \geq 0}$ converges in law to the process $(S(t)\,\P)_{t \geq 0}$, in the so-called Meyer-Zheng topology, where $(S(t))_{t \geq 0}$ is the Fisher-Wright diffusion with diffusion constant $D$ given by $1/D = \int_{\R_+} (1/r)\,\P(\d r)$. 
\end{abstract}



\section{Introduction}
\label{s.introduct}

The Moran model, describing the evolution via \emph{resampling} of a population of individuals carrying genetic types, is a workhorse in population genetics. It is tractable because it is mean-field, i.e., all individuals exchange type with each other in the same manner, and because it has a simple dual, namely, the death process on the positive integers. It has been the basis for much more sophisticated models incorporating evolutionary forces such as mutation, selection, recombination and migration (see \cite{DU08}). 


\subsection{The two-type Moran model}

{A} standard model for evolution of types in population genetics is the two-type Moran model (see \cite[Section 1.5]{DU08}). Consider a population of $N \in \N$ individuals. Each individual carries either type $\h$ or type $\s$, and at rate $1$ chooses an individual uniformly at random from the population (possibly itself) and adopts its type. Let $X^N(t)$ be the number of individuals with type $\h$ at time $t$. For every $N \in \N$,
\begin{equation}
\lim_{t \to \infty} \P(X^N(t) = 0 \text{ or } X^N(t) = N) = 1,
\end{equation}
because $\{0\}$ and $\{N\}$ are traps for $(X^N(t))_{t \geq 0}$: genetic diversity is lost forever when all individuals have type $\s$ or all individuals have type $\h$, respectively. 

To see what happens prior to trapping, scale space and time to get
\begin{equation}
Y^N(t) = \frac{1}{N} X^N(Nt),
\end{equation}
which is the fraction of individuals of type $\h$ at time $Nt$. It is well known that, if $\lim_{N\to\infty} Y^N(0) = y \in [0,1]$, then 
\begin{equation}
\label{genconv}
\lim_{N\to\infty} \cL[Y^N] = \cL[Y] \quad \text{in the Skorohod topology},
\end{equation}
where $Y^N = (Y^N(t))_{t \geq 0}$ and $Y=(Y(t))_{t \geq 0}$, and $\cL$ stands for law. The limit $Y$ is the \emph{Fisher-Wright diffusion}
\begin{equation}
{{\mathrm d}}Y(t) = \sqrt{Y(t)\,(1-Y(t))} \, {{\mathrm d}}W(t), \qquad Y(0) = y,
\end{equation}
where $W=(W(t))_{t \geq 0}$ is standard Brownian motion. The convergence in \eqref{genconv} is proven by showing that the infinitesimal generator $L^N$ of $Y^N$ converges to the infinitesimal generator $L$ of $Y$ on a dense set of test functions (see \cite[Chapter 4, Section 8]{EK86}). The latter is given by
\begin{equation} \label{eq:gen}
(Lf)(y) = y(1-y) f''(y), \qquad y \in [0,1],
\end{equation}
for twice differentiable test functions $f\colon\,[0,1] \to \R$. 

Let $\tau = \inf\{t \geq 0\colon\,Y(t) = 0 \text{ or }Y(t) = 1\}$. It is know that $\tau < \infty$ almost surely, i.e., genetic diversity is lost at time $\tau N + o(N)$ as $N\to\infty$.  


\subsection{The two-type Moran model with random resampling rates} 
\label{subsec: our contribution}

It is natural to allow for resampling rates that are \emph{random}. {Indeed, individuals typically live in different \emph{environments}, causing them to have different \emph{fitnesses} (e.g.\ due to lack or abundance of nutrients). In fact, the standard Moran model is an idealisation in which such natural variations are discarded. The situation is similar to what is seen in statistical physics, where disorder is introduced in the interaction between the constituent particles of a material, simply because almost all materials in nature are inhomogeneous at the microscopic level. The disorder typically affects the behaviour at the macroscopic level, in particular, the \emph{transport coefficients} that capture physical phenomena like diffusion, viscosity or compressibility.} 

In the present paper we consider a version of the two-type Moran model in which each individual carries a \emph{pre-assigned resampling rate}, drawn from a discrete probability distribution $\P$ on $\R_+ = (0,\infty)$. We will see that, in the large-population-size limit and after scaling time by the population size, the empirical distribution of the resampling rates of the individuals with type $\h$ converges to $\P$ times a Fisher-Wright diffusion with a \emph{diffusion constant} that depends on $\P$. The proof of this fact turns out to be somewhat delicate because the path of the empirical distribution \emph{contracts to the submanifold} $\{s\P\colon\,s \in [0,1]\}$ at a rate that diverges with the population size (so that convergence of the infinitesimal generator away from the submanifold fails). This contraction, which is absent in the standard Moran model, requires us to work with a particular path topology, called the \emph{Meyer-Zheng topology}, which roughly speaking induces weak convergence in ``time $\times$ space''. We may think of the latter as convergence outside a set of exceptional times.

Very little is known about population genetic models with disorder. The reason is that \emph{the disorder destroys the exchangeability of the individuals}, {especially in the quenched setting,} and therefore makes duality, another workhorse in population genetics, much more involved. In fact, in the present paper we do not resort to duality and work with \emph{Lyapunov functions} to control the contraction to the submanifold. How the contraction precisely occurs comes with a number of subtleties. Our proofs are restricted to discrete $\P$ because the submanifold depends on $\P$. Perturbation arguments are delicate because they displace the submanifold. 


\subsection{Main theorem} 
\label{sec:mr}

Throughout the sequel we assume that $\P$ satisfies the following assumption. 

\begin{assumption}
\label{ass}
$\mbox{}$
\begin{itemize}
\item[{\rm (1)}] 
$\P$ has countable support, i.e., $\P = \sum_{k\in\N} \mu_k \delta_{r_k}$ with $\mu=(\mu_k)_{k \in \N}$ a probability distribution on $\N$ and $(r_k)_{k\in\N}$ a sequence in $\R_+$. Moreover, $\sum_{k\in\N} \mu_k r_k^{-1} < \infty$.
\item[{\rm (2)}] 
Let $R=(R_i)_{i\in\N}$ be i.i.d.\ drawn from $\P$, i.e., $R$ has law $\P^{\otimes\N}$. For $N\in\N$, define
\begin{equation}
\label{NRdef}
\hat{n}^N_k =  \sum_{i \in [N]} 1_{\{R_i = r_k\}}, \quad k \in \N, \qquad N_R = \{k\in\N\colon\,\hat{n}_k^N >0\}. 
\end{equation}
Then
\begin{equation}
\label{eq:key}
\lim_{N\to\infty}\frac{(\sum_{k\in N_R} r_k)\,(\sum_{k\in N_R} r_k^{-2})}{N \min_{k\in N_R} r_k} = 0
\qquad \P^{\otimes\N}\text{-a.s.}
\end{equation}
\end{itemize}
\end{assumption} 
\noindent
Note that Assumption~\ref{ass} holds for all $\P$ with finite support and fails for all $\P$ with uncountable support. In Section~\ref{subsec:tail} we derive a sufficient condition on $\P$ with countable support guaranteeing that \eqref{eq:key} holds. For instance, it suffices that $\mu_k \asymp k^{-\chi}$, $k \to \infty$, with $\chi>2$ (right tail condition) and $\sum_{k\in\N} \mu_k\, r_k^{-\alpha} < \infty$, $\sum_{k\in\N} \mu_k\,r_k^{\beta} < \infty$ for $\alpha$ and $\beta$ sufficiently large depending on $\chi$ (left and right moment conditions). 

In what follows, we consider the \emph{quenched} process in which the resampling rates are chosen initially and are kept fixed while the process evolves. Let $R_i$ denote the random resampling rate of individual $i$. Note that $R=(R_i)_{i \in [N]}$ has law $\P^{\otimes [N]}$. The empirical distribution of the resampling rates is
\begin{equation}
\cE^N = \frac{1}{N} \sum_{i \in [N]} \delta_{R_i}. 
\end{equation}
Note that this is a probability distribution on $\R_+$ and that, $\P^{\otimes\N}$-a.s., $\cE^N$ converges to $\P$ as $N\to\infty$. Let $\tau_i(t)$ denote the type of individual $i$ at time $t$, and let 
\begin{equation}
Y^N(t) = \frac{1}{N} \sum_{i \in [N]} 1_{\{\tau_i(Nt) = \h\}} \delta_{R_i} 
\end{equation}
denote the \emph{empirical distribution of the resampling rates of the individuals with type $\h$ at time $Nt$}. Note that this is an element of $\cM_{\leq 1}$, the set of sub-probability distributions on $\R_+$. Our main result says that the process $Y^N=(Y^N(t))_{t \geq 0}$ converges in law, in the so-called \emph{Meyer-Zheng topology} (explained in Section~\ref{sec: defnot}), to a simple but interesting limiting process. We may rewrite $\cE^N$ as 
\begin {equation}
\label{nNkdef}
\cE^N = \sum_{k\in\N} n^N_k \delta_{r_k}, \qquad n^N_k = \frac{1}{N} {\hat{n}^N_k} = \frac1N  \sum_{i\in[N]} 1_{\{R_i=r_k\}}.
\end{equation} 
Note that $N_R =\{k \in \N\colon\, n^N_{k}>0\}$. Also note that $|N_R| \leq N$, and consequently the vector $n^N=(n^N_k)_{k \in \N}$ has at most $N$ non-zero elements. 

\begin{theorem} 
\label{thm:main}
Suppose that\/ $\P$ satisfies Assumption~\ref{ass}. If\/ $\lim_{N\to\infty} Y^N(0) = \nu$, then
\begin{equation} 
\lim_{N\to\infty} \cL[Y^N] = \cL[Y] \quad \text{in the Meyer-Zheng topology} \quad \P^{\otimes\N}\text{-a.s.}
\end{equation}
with $Y = (Y(t))_{t \geq 0}$ given by
\begin{equation}
Y(t) = S(t)\,\P,
\end{equation}
where $S=(S(t))_{t \geq 0}$ is the Fisher-Wright diffusion on $[0,1]$ given by
\begin{equation} \label{FWlim}
{{\mathrm d}}S(t) = \sqrt{D\,S(t)\,(1-S(t))}\, {{\mathrm d}}W(t), \qquad S(0) = \nu(\R_+),
\end{equation}
with diffusion constant $D$ defined as $1/D = \int_{\R_+} (1/r)\,\P({{\mathrm d}}r) = \sum_{k\in\N} \mu_k\,r_k^{-1}$.
\end{theorem}

\begin{remark}
{\rm Theorem \ref{thm:main} says that $Y^N=(Y^N(t))_{t \geq 0}$ collapses onto the submanifold $\{s\,\P\colon\,s \in [0,1]\}$ in the limit as $N\to\infty$ and on this submanifold performs an autonomous Fisher-Wright diffusion with a diffusion constant that is determined by $\P$. The contraction to the submanifold shows that $Y^N$ exhibits \emph{homogenisation}, i.e., the type and the resampling rate of the individuals are asymptotically independent as $N\to\infty$. The inverse of the diffusion constant $1/D$ represents the average time between two successive resamplings of an individual. Section \ref{subsec:tail} contains examples where Assumption~\ref{ass} holds. Assumption~\ref{ass} excludes $\P$ with uncountable support. We conjecture that Theorem \ref{thm:main} continues to hold in that setting, subject to appropriate tail and moment conditions, but we are unable to provide a proof. Along the way we will see what the hurdles are.}
 \end{remark}
 
{
\begin{remark}
{\rm The diffusion constant $D$ can be viewed a modulating an \emph{effective population size}, a notion that is important in population genetics (see e.g.\ \cite[Section 4.4]{DU08}). Since the time scale of the evolution is $N$, which is the population size, speeding up the evolution by a factor $D$ means that the effective population size is $DN$.}
\end{remark}
}
 
\begin{remark}
{\rm Our model can be interpreted as a Moran model with a particular kind of selection by viewing the resampling as done in pairs. Pick a pair of individuals, decide randomly which of the two individuals adopts the type of the other individual. In the standard Moran model the latter is decided uniformly at random. In our model it is decided with probabilities that depend on the inverses of the resampling rates of the two individuals.} 
\end{remark}

\begin{remark}
{\rm It is straightforward to extend Theorem~\ref{thm:main} to the Moran model with finitely many types. It is more challenging to include the Fleming-Viot model with infinitely many types. In the latter setting, it would be interesting to see what happens when the resampling rate distribution has no finite inverse moment at zero. Slow rates may affect the property that the state of the Fleming-Viot diffusion after an arbitrary positive time becomes atomic (a property referred to as ``coming down from infinity'').} 
\end{remark}

\paragraph*{Acknowledgements} 
SA was supported through a Knowledge Exchange Grant from ICTS. FdH was supported through NWO Gravitation Grant NETWORKS-024.002.003. AR was supported through Singapore Ministry of Education Academic Research Fund Tier 2 grant MOE2018-T2-2-076. The present work evolved out of the master thesis written by Jannetje Driessen at Leiden University completed on 25 May 2022. We also thank the International Centre for Theoretical Sciences (ICTS) for hospitality during the  ICTS-NETWORKS Workshop ``Challenges in Networks'' in January 2024.


\section{Definitions and notation}
\label{sec: defnot}

In Section~\ref{subsec: proc} we define the random process precisely and in Section~\ref{subsec: MZ} we recall the Meyer-Zheng topology. In Section~\ref{subsec:tail} we derive some tail estimates for the random resampling rates and exhibit sufficient conditions under which Assumption \ref{ass} holds.


\subsection{The random process} 
\label{subsec: proc}

Let $N \in \N$ be the size of the population. Individuals are labelled $[N]=\{1,\ldots,N\}$ and carry either type $\h$ or type $\s$. Each individual has a random resampling rate that is drawn from a probability distribution $\P$ on $\R_+$. The resampling rates are assigned \emph{independently} and before the process starts. {Recall the definition of $R = (R_i)_{i \in \N}$, the resampling rates assigned to the individuals, and the definition of $Y^N=(Y^N(t))_{t \geq 0}$, the empirical distribution of the resampling rates of the individuals with type $\h$ at time $Nt$, both introduced in Section \ref{sec:mr}.}

For $\Delta \in \{-1, +1\}$, $r \in \R_+$ and $y \in \cM_{\leq 1}$, put
\begin{equation} 
\label{eq:operT}
T^{r,\Delta}(y) = y + \frac{\Delta}{N}\delta_r.
\end{equation}
This operator encodes the event that an individual with resampling rate $r$ changes its type. Namely, when $\Delta = -1$ an individual with resampling rate $r$ and type $\h$ chooses an individual with type $\s$ and adopts its type, while when $\Delta = +1$ an individual with resampling rate $r$ and type $\s$ chooses an individual with type $\h$ and adopts its type. Define
\begin{equation}  
\label{eq:operR}
R^{r,-1}(y)\, {{\mathrm d}}r = r y({{\mathrm d}}r) [1-y(\R_+)], 
\qquad R^{r,+1}(y)\, {{\mathrm d}}r = r[\cE^N({{\mathrm d}}r) - y({{\mathrm d}}r)] y(\R_+).
\end{equation}
The change $\hat{T}^{r,-1}(y)$ happens at rate $\hat{R}^{r,-1}(y)$, because there are $y({{\mathrm d}}r)$ individuals with type $\h$ that can choose a new individual from the population at rate $r$. In order for such an individual to change to $\s$, it needs to choose an individual with type $\s$. There is a fraction $1-y(\R_+)$ individuals of type $\s$ in total, and the individual chooses one of them uniformly. The explanation for $T^{r,+1}(y)$ and $R^{r, +1}(y)$ is analogous. 

Note that the transition rates $R^{r,\Delta}(y)$ and the transition operators $T^{r,\Delta}(y)$ only depend on the current state $y$ of the process. Hence $Y^N$ is Markov. We can write down the infinitesimal generator $L^N$ of $Y$. Let $f\colon\,\cM_{\leq 1} \to \R$ be a bounded and continuous test function, and let $y \in \cM_{\leq 1}$. Then
\begin{equation}  
\label{eq:genL}
(L^Nf)(y) = \int_{\R_+} {{\mathrm d}}r \sum_{\Delta \in \{-1,+1\}} R^{r,\Delta}(y)[f(T^{r,\Delta}(y)) - f(y)] N^2.
\end{equation} 
The term $N^2$ comes from the time scaling by $N$ and the space scaling by $1/N$.


\subsection{The Meyer-Zheng topology}
\label{subsec: MZ}
 
For more background on what is written below, we refer the reader to \cite[Chapter III, no.\,40--46]{DM78}, \cite{MZ84},  and \cite[Appendix B]{GdHO21}. {The Meyer-Zheng topology is a topology on path-space that is weaker than the Skorohod topology. Intuitively, it captures weak convergence in `space $\times$ time', i.e., it allows for deviations of the path on exceptional sets of time. In our context these exceptional times are those where~$Y^N(t)$ could potentially be away from the submanifold $\{s\P\colon\,s \in [0,1]\}$.} 

The Meyer-Zheng topology assigns to each $\bar{\R}$-valued Borel measurable function $w=(w(t))_{t \geq 0}$ a probability measure on $[0,\infty] \times \bar{\R}$ defined by
\begin{equation}
\psi_w(A) = \int_{[0,\infty] \times \bar{\R}} 1_A(w(t)) \, \e^{-t} \, {{\mathrm d}}t,
\end{equation}
i.e., $\psi_w$ is the image under the mapping $t \to (t,w(t))$ of the probability measure $\e^{-t}{{\mathrm d}}t$. This image measure is called the \emph{pseudopath} associated with $w$, and is simply the occupation measure of $w$. We say that a sequence of pseudopaths induced by a sequence of c\`adl\`ag paths $(w_N)_{N \in \N}$ converges to a pseudopath induced by a c\`adl\`ag path $w$ if, for all continuous bounded function $f(t,w(t))$ on $[0,\infty] \times \bar{\R}$,

\begin{equation}
\lim_{N\to\infty} \int_{[0,\infty)} f(t, w_N(t))\ \e^{-t} \, {{\mathrm d}}t = \int_{[0,\infty)} f(t,w(t))\ \e^{-t} \, {{\mathrm d}}t.
\end{equation}
Since a pseudopath is a measure, convergence of pseudopaths is convergence of measures. Let $\Xi$ be the space of all pseudopaths. Endowed with the Meyer-Zheng topology, $\Xi$ is a Polish space {(see, for example, \cite[Section~4]{Kurtz1991a})}.

Let $v_N\colon\, [0,\infty) \to \cM_{\leq 1}$ denote the path of $Y^N$ and $v\colon\, [0,\infty) \to \cM_{\leq 1}$ the path of $Y$. For $f \in C_b(\Xi)$, define
\begin{equation}
\psi_{v_N}(f) = \int_{[0,\infty)} f(t, v_N(t))\ \e^{-t} \, {{\mathrm d}}t, \qquad \psi_v(f) = \int_{[0,\infty)} f(t, v(t))\ \e^{-t} \, {{\mathrm d}}t.
\end{equation}
Suppose that $f(t,x) = 1_{t \in [a,b]}\bar{f}(x)$, with $\bar{f}\colon\, \cM_{\leq 1} \to \R$ bounded and continuous and $0 \leq a < b < \infty$. The set of all functions of this form generates $C_b(\Xi)$. Write
\begin{equation}
\psi_{v_N}(f) = \int_a^b \bar{f}(v_N(t))\ \e^{-t} \, {{\mathrm d}}t, \qquad \psi_v(f) = \int_a^b \bar{f}(v(t))\ \e^{-t} \, {{\mathrm d}}t.
\end{equation}
To prove convergence in the Meyer-Zheng topology, we need to show that
\begin{equation}
\lim_{N \to \infty} \E[\psi_{v_N}(f)] = \E[\psi_{v}(f)] \qquad \forall\,f \in C_b(\Xi),
\end{equation}
which is equivalent to
\begin{equation} 
\lim_{N \to \infty} \E \left[ \int_a^b \left|\bar{f}(Y^N(t)) - \bar{f}(Y(t))\right|\,\e^{-t} \, {{\mathrm d}}t \right] = 0
\qquad \forall\,\bar{f} \in C_b(\cM_{\leq 1}),\, 0 \leq a < b < \infty.
\end{equation}

Suppose that $\bar{f}$ is Lipschitz continuous with Lipschitz constant $C \in \R$ for a given distance $d$ on $\cM_{\leq 1}$. By the Stone-Weierstrass theorem, the space of polynomial functions on the set of closed intervals in $[0,1]^N$ is dense in the space of continuous functions on $[0,1]^N$. Furthermore, the space of Lipschitz functions contains the space of polynomial functions. This means that any result for Lipschitz functions can be generalised to continuous $\bar{f}$. We find that
\begin{equation}
\begin{aligned}
\E\left[\int_a^b \big|\bar{f}(Y^N(t)) - \bar{f}(Y(t))\big|\,\e^{-t} \, {{\mathrm d}}t \right]
&\leq C\,\E\left[ \int_a^b d(Y^N(t),Y(t))\, \e^{-t} \, {{\mathrm d}}t \right]\\ 
&= C  \int_a^b \E\big[d(Y^N(t),Y(t))\big]\, \e^{-t} \,{{\mathrm d}}t,
\end{aligned}
\end{equation}
where we use Fubini's theorem to interchange the integral and the expectation. Thus, if we manage to prove that
\begin{equation}
\label{dzerotarget}
\lim_{N \to \infty} \E\big[d(Y^N(t),Y(t))\big] = 0 \text{ uniformly in $t$ over compact subsets of $(0,\infty)$},
\end{equation}
then we have proven that $Y^N$ converges to $Y$ in the Meyer-Zheng topology.  {We will choose for $d$ a suitable $l^2$-metric and establish \eqref{dzerotarget} in Section \ref{sec:deflem}.}


\subsection{Examples satisfying Assumption \ref{ass}}
\label{subsec:tail}

In this section we will identify examples satisfying Assumption \ref{ass} using tail and moment estimates. We will need  the following conditions:
\begin{itemize}
\item[(I)]
There are $\alpha,\beta \in (0,\infty)$ such that
\begin{equation}
\label{cond1}
\sum_{k\in\N} \mu_k\, r_k^{-\alpha} < \infty, \qquad \sum_{k\in\N} \mu_k\,r_k^{\beta} < \infty. 
\end{equation}
\item[(II)]
There are $\gamma,\delta \in (0,1]$ such that
\begin{equation}
\label{cond2}
\limsup_{\epsilon \downarrow 0} \epsilon^{-(1-\gamma)} \sum_{k\in\N} \mu_k\,1_{\{\mu_k \leq \epsilon\}} < \infty, 
\qquad
\limsup_{\epsilon \downarrow 0} \epsilon^{\delta} \sum_{k\in\N} 1_{\{\mu_k > \epsilon\}} < \infty. 
\end{equation}
\item[(III)]
There is a $\chi>1$ such that 
\begin{equation}
\label{cond3}
\mu_k\asymp k^{-\chi}, \qquad k \to \infty.
\end{equation} 
\end{itemize}

\begin{lemma}
\label{properties}
{Recall the definition of $(R_i)_{i \in \N}$ from Section \ref{sec:mr}. For $N \in \N$, let}
\begin{equation}
M^-_N = \min_{i \in [N]} R_i, \quad M^+_N = \max_{i \in [N]} R_i, \quad 
Z_N = {|N_R|.}
\end{equation}
Write $\P$ and $\E$ to denote probability and expectation with respect to the law $\P^{\otimes\N}$ of $(R_i)_{i\in\N}$.\\[0.2cm] 
If Condition\ {\rm(I)} holds, then
\begin{itemize}
\item[{\rm (1)}]
$\lim_{\epsilon \downarrow 0} \liminf_{N\to\infty} \P(M^-_N N^{1/\alpha} \geq \epsilon  ) =1$.
\item[{\rm (2)}]
$\lim_{\epsilon \downarrow 0} \liminf_{N\to\infty} \P( M^+_NN^{-1/\beta} \leq \epsilon^{-1}) =1$.
\end{itemize}
If Condition\ {\rm(II)} holds, then
\begin{itemize}
\item[{\rm (3)}]
$\lim_{\epsilon \downarrow 0} \liminf_{N\to\infty} \P( Z_NN^{-(\gamma \vee \delta)} \leq \epsilon^{-1}) = 1$.
\end{itemize}
If Condition\ {\rm(III)} holds, then
\begin{itemize}
\item[{\rm (4)}] 
$\lim_{\epsilon \downarrow 0} \liminf_{N\to\infty} \P( Z_NN^{-\frac{\upsilon}{\chi -1}} \leq \epsilon^{-1}) = 1\,\, \forall\, v>1$.
\end{itemize}
\end{lemma}

\begin{proof}
(1) Estimate, for $\epsilon>0$,
\begin{equation}
\begin{aligned}
&\P(M^-_N \geq \epsilon N^{-1/\alpha}) = \P\big(R_i \geq \epsilon N^{-1/\alpha} \,\, \forall\, i \in [N]\big) 
= \P(R_1 \geq \epsilon N^{-1/\alpha})^N\\ 
&= \big(1-\P((1/R_1)^\alpha > \epsilon^{-\alpha} N)\big)^N \geq \left(1-\epsilon^\alpha N^{-1} \E[(1/R_1)^\alpha]\right)^N
\to \exp(-\epsilon^\alpha \E[(1/R_1)^\alpha]).
\end{aligned}
\end{equation}
By \eqref{cond1}, $\E[(1/R_1)^\alpha]<\infty$, and so the claim follows after letting $\epsilon \downarrow 0$.

\medskip\noindent 
(2) Estimate, for $\epsilon>0$,
\begin{equation}
\begin{aligned}
&\P(M^+_N \leq \epsilon^{-1} N^{1/\beta}) = \P\big(R_i \leq \epsilon^{-1} N^{1/\beta} \,\, \forall\, i \in [N]\big) 
= \P(R_1 \leq \epsilon^{-1} N^{1/\beta})^N\\ 
&= \big(1-\P(R_1^\beta > \epsilon^{-\beta} N)\big)^N \geq \left(1- \epsilon^\beta N^{-1}\E[(R_1)^\beta]\right)^N
\to \exp(-\epsilon^\beta \E[(R_1)^\beta]).
\end{aligned}
\end{equation}
By \eqref{cond1}, $\E[(R_1)^\beta]<\infty$, and so the claim follows after letting $\epsilon \downarrow 0$.

\medskip\noindent 
(3) Estimate
\begin{equation}
\label{Neps}
\P(N^{-(\gamma \vee \delta)} {Z_N} \leq \epsilon^{-1}) = 1 - \P(N^{-(\gamma \vee \delta)}  {Z_N} > \epsilon^{-1})
\geq 1 - \epsilon N^{-(\gamma \vee \delta)} \E[Z_N].
\end{equation}
Write $Z_N = \sum_{i \in [N]} 1_{\{R_i \neq R_j\,\forall\, 1 \leq j < i\}}$, and compute
\begin{equation}
\begin{aligned} \label{eq:EZ}
&\E[Z_N] = \sum_{i \in [N]} \P(R_i \neq R_j\, \forall\, 1 \leq j < i)\\ 
&= \sum_{i \in [N]} \sum_{k \in \N}
\P(R_i = r_k)\, \P(R_i \neq R_j\, \forall\, 1 \leq j < i \mid R_i = r_k)\\
&= \sum_{i \in [N]} \sum_{k \in \N} \mu_k (1-\mu_k)^{i-1} = \sum_{k \in \N} \mu_k\,\frac{1-(1-\mu_k)^N}{1-(1-\mu_k)}
= \sum_{k \in \N} [1-(1-\mu_k)^N].
\end{aligned}
\end{equation}
Define
\begin{equation}
A_N = \sum_{k \in \N} \mu_k\,1_{\{\mu_k \leq 1/N\}}, \qquad B_N = \sum_{k \in \N} 1_{\{\mu_k > 1/N\}}.
\end{equation}
Then 
\begin{equation}
\E[Z_N] \leq  c_N N A_N + B_N
\end{equation} 
with $c_N = \sup_{x \in (0,1/N]} [-\frac{1}{x}\log (1-x)] \to 1$ as $N\to\infty$. By \eqref{cond2}, $\limsup_{N \to \infty} N^{1-\gamma} A_N $ $< \infty$ and $\limsup_{N \to \infty} N^{-\delta} B_N < \infty$. Hence $\limsup_{N\to\infty} N^{-(\gamma \vee \delta)} \E[Z_N] < \infty$, and so the claim follows after letting $\epsilon \downarrow 0$ in \eqref{Neps}.

\medskip\noindent 
(4) The proof of this part is inspired by the proof presented in \cite[Lemma 5]{ZK01}.  Let $\upsilon >1$. Then $\frac{\upsilon}{\chi -1}> \frac{1}{\chi-1}$, and hence $\chi-\frac{\chi -1}{\upsilon}>1$. Let $N_0 = \lfloor {N^\frac{\upsilon}{\chi-1}} \rfloor$. 
Then {from \eqref{eq:EZ} we have}
\begin{equation} \label{eq:zkstep}
\begin{aligned}
{\E[Z_N]}& = \sum_{k \in \N} [1-(1-\mu_k)^N]\\ 
&= \sum_{{1 \leq k \leq  N_0}} [1-(1-\mu_k)^N] + \sum_{k \geq N_0+1} [1-(1-\mu_k)^N]\\
&\leq N_0 +\sum_{k \geq N_0+1} [1-(1-\mu_k)^N]. 
\end{aligned}
\end{equation}
Note that, for ${k} \geq N_0+1= \lfloor {N^\frac{\upsilon}{\chi-1}} \rfloor +1$, {by using \eqref{cond3} we see that} there are $C_1,C_2>0$ such that
\begin{equation}
1-(1-\mu_k)^N  \leq 1-\left(1- C_1\frac{1}{k^\chi}\right)^{k^{\frac{\chi -1}{\upsilon}}} 
\leq C_2 \frac{1}{k^{\chi-\frac{\chi -1}{\upsilon}}}.
\end{equation}
{Since $\chi-\frac{\chi -1}{\upsilon} >1$, this implies that $\sum_{k \geq N_0+1} [1-(1-\mu_k)^N] < \infty$. Using the latter and the definition of $N_0$ in \eqref{eq:zkstep}, we see that there is a $C_3 >0$ such that $E[Z_N] \leq C_3 N^{\frac{\chi}{\upsilon -1}}$. Using this bound and the same argument as in \eqref{Neps}, we get
\begin{equation*}
\label{eq:Neps1} 
\P(N^{-\frac{\upsilon}{\chi-1}} {Z_N} \leq \epsilon^{-1}) \geq 1 - \epsilon N^{-\frac{\upsilon}{\chi-1}} \E[Z_N] \geq 1- \epsilon.\end{equation*} 
}
The claim follows after letting $\epsilon \downarrow 0$.
\end{proof}

\begin{corollary}
\label{momcond}
Assumption~\ref{ass} is satisfied if either of the following is true: 
\begin{itemize} 
\item[(a)] 
Conditions\ {\rm(I)} and\ {\rm(II)} hold with 
\begin{equation}
\label{parass}
2(\gamma \vee \delta) + \frac{1}{\beta} < 1 - \frac{3}{\alpha}. 
\end{equation}
\item[(b)] 
Conditions\ {\rm(I)} and\ {\rm(III)} hold with 
\begin{equation}
\label{parass2}
\frac{1}{\chi-1} + \frac{1}{\beta} < 1 - \frac{3}{\alpha}. 
\end{equation}
\end{itemize}
\end{corollary}

\begin{proof}
Note that, {by the definition of $N_R$ from Section \ref{sec:mr} and $M_N^-$, $M_N^+$ from Lemma~\ref{properties}, we have $\P^{\otimes\N}\text{-a.s.}$}
\begin{equation}
\label{201}
\frac{\sum_{k\in N_R} r_k\sum_{k\in N_R} r_k^{-2}}{N \min_{k\in N_R} r_k}
\leq \frac{Z_N^2 M_N^+}{N(M_N^-)^3}.
\end{equation}
(a) Note that Lemma~\ref{properties}(1) implies that a.s.\ there exist $\epsilon_1>0$ and $N_1$ such that
\begin{equation}
M^-_N N^{1/\alpha} \geq \eps_1 \qquad \forall\,N>N_1.
\end{equation}
Likewise, we can find $N_2,N_3$ and $\epsilon_2,\epsilon_3$ by Lemma~\ref{properties}(2,3). It follows that a.s.\ for $N_0 = \max\{N_1,N_2,N_3\}$ and $\epsilon_0=\min\{\epsilon_1,\epsilon_2,\epsilon_3\}>0$,
\begin{equation}
\frac{Z_N^2 M_N^+}{N(M_N^-)^3} \leq \frac{N^{2(\gamma\vee\delta)} N^{1/\beta}}
{\epsilon_0^6 N^{1-3/\alpha}} \qquad \forall\, N>N_0.
\end{equation}
From this it easily follows that the right-hand side of \eqref{201} converges to zero a.s.\\
(b) Note that, as in part (a), using Lemma~\ref{properties}(1,2,4) we have that a.s.\ there are $\tilde{N}$ and $\tilde{\epsilon}$ such that
\begin{equation}
\frac{Z_N^2 M_N^+}{N(M_N^-)^3} \leq \frac{N^{\frac{\upsilon}{\chi-1}} N^{1/\beta}}{\tilde{\epsilon}^6 N^{1-3/\alpha}}
\qquad \forall\,N>\tilde{N}.
\end{equation}
From this it easily follows that the right-hand side of \eqref{201} converges to zero a.s. (recall that $v>1$ is arbitrary).
\end{proof}

The example given below Assumption~\ref{ass} follows from \eqref{parass2}. We close with two further examples for which \eqref{parass} and \eqref{parass2} are met.

\begin{example} 
{\rm (I) Pick $r_k=k$ and $\mu_k=p(1-p)^{k-1}$ for $k\in\N$ and some $0<p<1$. Then \eqref{cond1} holds for any $\alpha,\beta<\infty$ and \eqref{cond2} holds for any $\gamma,\delta>0$. Hence \eqref{parass} is met by taking $\alpha,\beta$ large enough and $\gamma,\delta$ small enough.\\
(II) Pick $r_k=k$ and $\mu_k=C(\chi) k^{-\chi}$ for $k\in\N$ and some $\chi>3$. Then \eqref{cond1} holds for any $\alpha<\infty$ and $\beta < \chi-1$ and \eqref{cond3} holds as well. Hence \eqref{parass2} is met by taking $\alpha >0$ large enough and $\beta <\chi-1$ close enough to $\chi-1$.}
\end{example}


\section{Proof of the main theorem} 

In Section~\ref{sec:deflem} we state a key lemma (Lemma~\ref{fourlemma} below) that implies Theorem~\ref{thm:main}. In Sections~\ref{sec:L1}--\ref{sec:L4} we prove this lemma. The proof is based on the construction of a \emph{Lyapunov function} together with a \emph{projection argument}, showing that the scaled dynamics rapidly contracts to the submanifold $\{s\,\P\colon\,s \in [0,1]\}$. Outside the submanifold the generator of the scaled dynamics diverges as $N\to\infty$. We show that inside the manifold the generator converges to that of the Fisher-Wright diffusion in \eqref{FWlim}.  Throughout this section, Assumption~\ref{ass} is in force.


\subsection{Definitions and a key lemma}
\label{sec:deflem}

Recall that we denoted by  $(R_i)_{i\in\N}$ the sequence of independent random variables with distribution $\P$, and  $\cE^N$ the empirical distribution of the first $N$ random variables, i.e., 
\begin {equation}
\cE^N = \sum_{k\in\N} n^N_k \delta_{r_k}, \qquad n^N_k = N^{-1} \sum_{i\in[N]} 1_{\{R_i=r_k\}}.
\end{equation} 
{Recall $N_R$ from Section \ref{sec:mr}}. The measure-valued process $Y^N$ can be represented as a vector-valued process $Y^N = (Y^N_k)_{k \in \N}$, also with at most $N$ non-zero elements. Each $Y^N_k$ may be viewed as the fraction of $Y^N$ that have the resampling rate $r_k$, and thus $Y^N_k >0$ if and only if $k \in N_R$.

Convergence of Markov processes is usually proven by showing that the infinitesimal generators converge (see \cite[Chapter 4, Section 8]{EK86}. For convenience we first reformulate the infinitesimal generator of $Y^N$. Let $f\colon\,\R^\N\to\R$ be three times partially differentiable, and let
\begin{equation}
\abs{f}_3 = \sup_{i,j,k\in\N} \left\| \frac{\partial^3}{\partial z_i\partial z_j\partial z_k} f \right\|_\infty.
\end{equation}
For $k\in\N$, let $X_{r_k}(t)$ be the number of individuals of type {$1$} with resampling rate $r_k$ at time $t$, and 
\begin{equation}
 Y_k^N(t) = \frac{X_{r_k}^N(Nt)}{N}
\end{equation}
the fraction of individuals at time $Nt$ of type $\h$ with resampling rate $r_k$. The random process $Y^N = (Y^N(t))_{t \geq 0}$ can now be viewed as $Y^N(t) = (Y^N_k(t))_{k\in\N}$ living on the state space
\begin{equation}
Q = \bclc{ y = (y_k)_{k\in\N} \colon \text{$0 \leq y_k \leq n_k^N\,\,\forall\,k\in\N$}}.
\end{equation}

For $k\in\N$, we can view \eqref{eq:operT} as
\begin{equation}
T^{k, \Delta}y = y + \Delta \frac{e_k}{N},
\end{equation}
with $e_k$ the unit vector in direction $k$, and the transition rates in \eqref{eq:operR} as 
\begin{equation}
R^{k, -1}(y) = r_k y_k \left( \sum_{l \in \N} (n_l^N - y_l) \right),  \qquad
R^{k, +1}(y) = r_k (n_k^N - y_k) \left( \sum_{l \in \N} y_l \right), 
\end{equation}
and the generator $L_N$ in \eqref{eq:genL} as 
\begin{equation}
(L_Nf)(y) = \sum_{k\in\N} \sum_{\Delta \in \{-1,+1\}} R^{k, \Delta}(y) \left[f(T^{k, \Delta}y) - f(y)\right] N^2.
\end{equation}
A straightforward Taylor expansion yields
\begin{equation}
\label{LNdef}
\begin{aligned}
(L^N f)(y) &= -N \sum_{k\in N_R} r_k \left(y_k - n^N_k \sum_{l \in N_R} y_l\right) f_k(y)\\
&\quad + \frac{1}{2} \sum_{k\in N_R} r_k \left(y_k - 2y_k \sum_{l \in N_R} y_l 
+ n^N_k \sum_{l \in N_R} y_l\right) f_{kk}(y) + \bigo\bbbclr{N^{-1}\abs{f}_3\sum_{k\in N_R}r_k},
\end{aligned}
\end{equation}
where $f_k(y) = (\partial/\partial y_k) f(y)$  and $f_{kk}(y) = (\partial/\partial y_k)^2 f(y)$.

We see that the first term in the infinitesimal generator of the process $Y^N$ has a prefactor $N$, and therefore diverges as $N \to \infty$ on most of the state space. However, we will be able to show that $Y^N$ is driven to a subspace on which the first term vanishes and the infinitesimal generator does converge. To account for this properly, we will prove that $Y^N$ converges in the Meyer-Zheng topology (from Section \ref{subsec: MZ}). We need to introduce some further quantities. Let
\begin{equation}
D^N = \left(\sum_{k\in N_R} \frac{n^N_k}{r_k}\right)^{-1}, \qquad D = \left(\sum_{k\in\N}\frac{\mu_k}{r_k}\right)^{-1}.
\end{equation} 
By Assumption~\ref{ass}, the strong law of large numbers gives (recall \eqref{nNkdef})
\begin{equation}
\label{100}
\lim_{N\to\infty} D^N = D \quad \P^{\otimes\N}\text{-a.s.}
\end{equation}
Define 
\begin{equation}
S^N(t) = \sum_{k\in N_R} Y^N_k(t), \qquad \˘S^N(t) = D^N \sum_{k\in N_R} \frac{1}{r_k} Y^N_k(t).
\end{equation}
Let $S=(S(t))_{t \geq 0}$ be a Fisher-Wright diffusion with diffusion constant $D$, defined on the same probability space as $Y^N$ {(we will specify the  coupling later)}. Using the triangle inequality, we have
\begin{equation}
\label{triangle}
\begin{aligned}
\bnorm{Y^N(t)-S(t)\P}_2 &\leq \bnorm{Y^N(t)-S^N(t)n^N}_2 + \bnorm{S^N(t)n^N-\˘S^N(t)n^N}_2\\
&\quad + \bnorm{\˘S^N(t)n^N-\˘S^N(t)\P}_2 + \bnorm{\˘S^N(t)\P - S(t)\P}_2,
\end{aligned}
\end{equation}
where $\|\cdot\|_2$ is the $\ell_2$-norm for vectors. We will bound the right-hand side via a sequence of lemmas summarised as follows. 

\begin{lemma} 
\label{fourlemma}
$\P^{\otimes\N}$-a.s.\ the expectation of each of the four terms in the right-hand side of \eqref{triangle} tends to zero as $N\to\infty$ uniformly in $t$ on compacts in $(0,\infty)$, where for the last term this holds for an appropriate coupling of $\˘S^N=(\˘S^N(t))_{t \geq 0}$ and $S=(S(t))_{t \geq 0}$.
\end{lemma}

Lemma~\ref{fourlemma} implies \eqref{dzerotarget} with $Y(t)=S(t)\P$ and $d$ the $L^2$-norm on $\cM_{\leq 1}$. It therefore settles the proof of Theorem \ref{thm:main}.


\subsection{First term}
\label{sec:L1}

\begin{proof}
The idea behind the proof is to find a Lyapunov function for the infinitesimal generator that rapidly decays to zero (for background on Lyapunov functions, see \cite[Chapter 4.6]{M08}). To this end, let $(c_k)_{k\in\N}$ be a sequence in $\R_+$ (to be determined later), and let $h\colon\,Q \to \R$ be defined by
\begin{equation}
h(y) = \sum_{m \in N_R} c_m \bbclr{y_m - n^N_m \sum_{l \in N_R} y_l}^2.
\end{equation}
We have
\begin{equation}
h_k(y)  = 2 \sum_{m \in N_R} c_m \left(\delta_{mk} - n^N_m\right) \bbclr{y_m - n^N_m \sum_{l \in N_R} y_l}
\end{equation}
and
\begin{equation}
h_{kk'}(y) = 2 \sum_{m \in N_R} c_m \left(\delta_{mk} - n^N_m\right) \left(\delta_{mk'} - n^N_m\right).
\end{equation}
Substituting these expressions into \eqref{LNdef}, we obtain 
\begin{equation}
\begin{aligned}
\mel (L^N h)(y)\\
&= - 2N \sum_{k \in N_R} c_k r_k \left(y_k - n^N_k \sum_{l \in N_R} y_l\right)^2\\
&\quad + 2N \sum_{k \in N_R} r_k \left(y_k - n^N_k \sum_{l \in N_R} y_l\right) 
\sum_{m \in N_R} c_m n^N_m \left(y_m - n^N_m \sum_{l \in N_R}  y_l\right) \\
&\quad + \sum_{k \in N_R} r_k \left(y_k - 2y_k \sum_{l \in N_R} y_l + n^N_k \sum_{l \in N_R} y_l\right) 
\sum_{m \in N_R}  c_m \left(\delta_{mk} - n^N_m\right) \left(\delta_{mk} - n^N_m\right).
\end{aligned}
\end{equation}
With the choice $c_k = 1/n^N_k$ for $k \in N_R$, the second term vanishes and, for $k \in N_R$,
\begin{equation}
\sum_{m \in N_R}  c_m \left(\delta_{mk} - n^N_m\right) \left(\delta_{mk} - n^N_m\right) 
=  \sum_{\substack{m\in N_R,\\ m \neq k}} n^N_m + \frac{1}{n^N_k}(1-n^N_k)^2 = \frac{(1-n^N_k)}{n^N_k}.
\end{equation}
Thus,
\begin{equation}
\begin{aligned}
(L^N h)(y) 
&= -2N \sum_{k \in N_R} \frac{r_k}{n^N_k} \left(y_k - n^N_k \sum_{l \in N_R} y_{l}\right)^2\\
&\qquad + \sum_{k \in N_R} \frac{r_k}{n^N_k} \bbbclr{y_k\bbclr{1- \sum_{l \in N_R} y_l} + (n^N_k - y_k) 
\sum_{l \in N_R} y_l} (1 - n^N_k)\\
&\leq -2N M^-_N h(y) + 2 \Sigma_N,
\end{aligned}
\end{equation}
where 
\begin{equation}
\label{MNSigNdef}
M^-_N = \min_{k \in N_R} r_k, \qquad \Sigma_N = \sum_{k \in N_R} r_k, 
\end{equation}
and we use the fact that $0\leq y_k\leq n^N_k$. Hence, putting 
\begin{equation}
g(t) = \E_y[h(Y^N(t))],
\end{equation} 
where $\E_y$ denotes expectation conditional on $Y^N(0)=y$, we obtain
\begin{equation}
g'(t) \leq -2N M^-_N g(t) + 2 \Sigma_N,
\end{equation}
and so, solving the differential inequality, we obtain
\begin{equation}
\label{bound}
\begin{aligned} 
\mel\E_y\big[\|Y^N(t)-S^N(t) n^N\|_2^2\big]\\
&= \E_y\left[\sum_{k \in N_R} \bbclr{Y^N_k(t) - n^N_k \sum_{{l \in N_R}} Y^N_l(t)}^2\right] \\
&\leq \E_y\left[\sum_{{k \in N_R}} \frac{1}{n^N_k}\bbclr{Y^N_k(t) - n^N_k \sum_l Y^N_l(t)}^2\right]  \\
&= g(t) \leq \frac{\Sigma_N}{N M^-_N} + \bbclr{g(0) -\frac{\Sigma_N}{N M^-_N}}\e^{-2N M^-_N t}.
\end{aligned}
\end{equation}
By Assumption~\ref{ass}, the first term in the right-hand side of \eqref{bound} converges to zero as $N\to\infty$ $\P^{\otimes\N}$-a.s. For the second term, note that
\begin{equation}
-n_k^N \leq  y_k-n^N_k \sum_{l \in N_R} y_l \leq n_k^N, \qquad k \in N_R,
\end{equation}
which implies
\begin{equation}
g(0) = \sum_{k \in N_R} \frac{1}{n^N_k} \bbclr{y_k-n^N_k \sum_{l \in N_R} y_l}^2 \leq 1. 
\label{eq:nkykb}
\end{equation} 
Since $\e^{-x} \leq x^{-1}$, $x>0$, and $\limsup_{N\in\N} (1/\Sigma_N) < \infty$ $\P^{\otimes\N}$-a.s., it follows that also the second term in the right-hand side of \eqref{bound} converges to zero as $N\to\infty$ $\P^{\otimes\N}$-a.s.\ and uniformly in $t$ on compacts in $(0,\infty)$. 
\end{proof}


\usetikzlibrary{calc,angles,quotes,arrows}
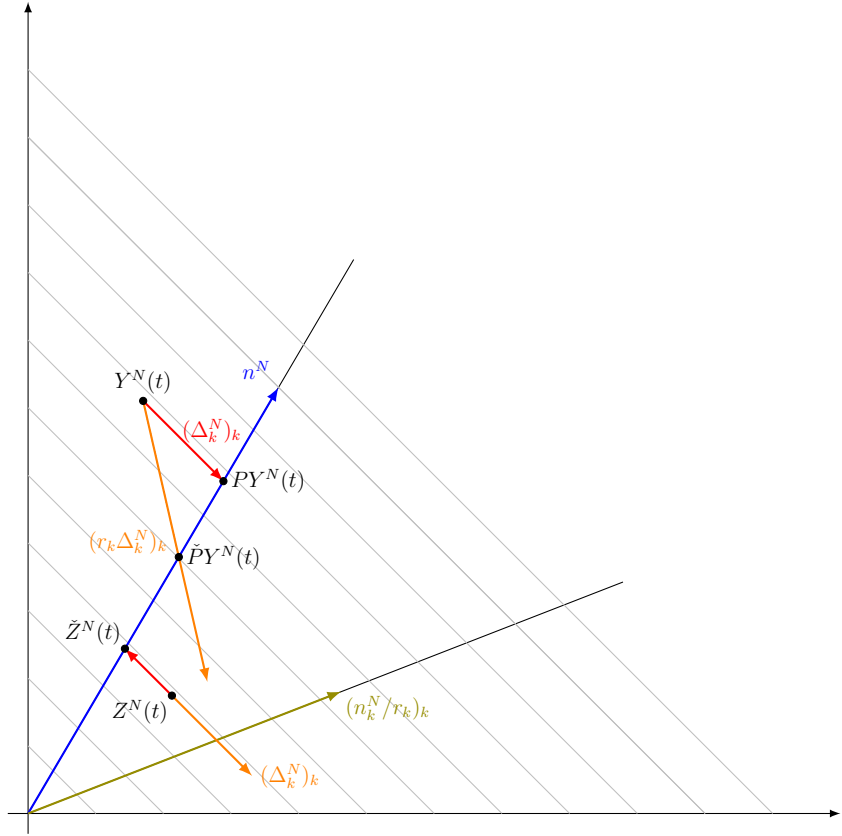
\begin{figure}[htbp]
\centering
\begin{tikzpicture}[>= latex, scale=9, every node/.style={scale=.7}]
\coordinate (origin) at (0,0);
\def\nx{0.37}
\def\ny{0.63}
\def\rx{0.8}
\def\ry{3.5}
\def\YNx{0.17}
\def\YNy{0.61}
\coordinate (nk) at (\nx,\ny);
\coordinate (YN) at (\YNx,\YNy);
\coordinate (ZN) at (\YNx/\rx,\YNy/\ry);
\coordinate (ZcheckN) at (\nx*\YNx/\rx+\nx*\YNy/\ry,\ny*\YNx/\rx+\ny*\YNy/\ry);
\pgfmathsetmacro{\SN}{\YNx+\YNy}
\coordinate (PY) at ($(origin)!\SN!(nk)$);
\newdimen\PYcX
\newdimen\PYcY
\pgfextractx{\PYcX}{\pgfpointanchor{PY}{center}}
\pgfextracty{\PYcY}{\pgfpointanchor{PY}{center}}
\pgfmathsetmacro{\PYx}{\PYcX/28.45274}
\pgfmathsetmacro{\PYy}{\PYcY/28.45274}
\pgfmathsetmacro{\DN}{1/(\nx/\rx+\ny/\ry)}
\pgfmathsetmacro{\SNcheck}{\YNx/\rx+\YNy/\ry}
\coordinate (PcheckY) at ($(origin)!\DN*\SNcheck!(nk)$);
\draw[->] (-0.03,0)--(1.2,0) node[right]{};
\draw[->] (0,-0.03)--(0,1.2) node[above]{};
\draw[black,thin] (origin) -- ($(origin)!1.3!(nk)$);
\draw[black,thin] (origin) -- ($(origin)!1.9!(\nx/\rx,\ny/\ry)$);
\foreach \i in {0.1,0.2,...,1.1,}{\draw[lightgray] (0,\i) -- (\i,0);}
\draw[thick,blue,->] (origin) -- (nk) node[at end,above left]{$n^N$};
\draw[thick,olive,->] (origin) -- (\nx/\rx,\ny/\ry) node[at end, below right=-2pt]{$(n^N_k/r_k)_k$};
\draw[thick,red,->] (YN) -- (PY) node[midway,above right=-5pt]{$(\Delta^N_k)_k$};
\draw[thick,red,->] (ZN) -- (ZcheckN) node[midway,above right=-5pt]{};
\draw[thin,black,densely dotted] (YN) -- (PcheckY);
\draw[thick,orange,->] (YN) -- ++(\PYx*\rx - \YNx*\rx,\PYy*\ry - \YNy*\ry) node[midway,left] {$(r_k\Delta^N_k)_k$};
\draw[thick,orange,->] (ZN) -- ++(\PYx - \YNx,\PYy - \YNy) node[at end, right] {$(\Delta^N_k)_k$};
\filldraw[black] (YN) circle(.15pt) node[above] {$Y^N(t)$};
\filldraw[black] (ZN) circle(.15pt) node[below left=-2pt] {$Z^N(t)$};
\filldraw[black] (ZcheckN) circle(.15pt) node[above left=-2pt] {$\check{Z}^N(t)$};
\filldraw[black] (PY) circle(.15pt) node[right] {$PY^N(t)$};
\filldraw[black] (PcheckY) circle(.15pt) node[right] {$\check{P}Y^N(t)$};
\end{tikzpicture}
\caption{\small {Illustration of the relationship between the empirical distribution process $Y^N(t)$, its projection $PY^N(t)$ onto the submanifold spanned by $n^N$, the adjusted projection $\check{P}Y^N(t)$ aligning with the drift, and transformed vectors $Z^N(t)$ and $\check{Z}^N(t)$. The figure demonstrates how the projection $PY^N(t)$ and drift are misaligned, making it necessary to work with the transformed processes $Z^N(t)$ and $\check{Z}^N(t)$, for which drift and projection coincide, which is necessary to prevent any `side-ways wind' close to the manifold. See splitting \eqref{triangle}, and also \eqref{proj}, \eqref{zcheck}, \eqref{del} and Remark~\ref{rem:proj}.}}
\label{fig1}
\end{figure}

\subsection{Second term}
\label{sec:L2}

\begin{proof} 
Define the maps $P$ and $\˘P$ from $Q$ to $Q$ by
\begin{equation}
\label{proj}
Py = \bbbclr{\,\sum_{k \in N_R} y_k}n^N,
\qquad
\˘Py = \bbbclr{D^N\sum_{k \in N_R} \frac{1}{r_k} y_k} n^N.
\end{equation}
It is straightforward to check that both $P$ and $\˘P$ are idempotent maps that project $Q$ onto the line spanned by the vector $n^N$. Hence, $Py = \˘PPy$, and so
\begin{equation}
\label{estterm2}
\begin{aligned}
\bnorm{S^N(t)n^N-\˘S^N(t)n^N}_2
& = \bnorm{PY^N(t)-\˘PY^N(t)}_2 = \bnorm{\˘P(PY^N(t)-Y^N(t))}_2\\
& \leq \norm{\˘P}_{\mathrm{op}}\bnorm{PY^N(t)-Y^N(t)}_2.
\end{aligned}
\end{equation}
Now, 
\begin{equation} \label{eq:npbr}
\bnorm{\˘Py}_2 = D^N\bbabs{\sum_{k \in N_R} \frac{y_k}{r_k}} \norm{n^N}_2 \leq D^N \,
\left(M_N^{(-2)}\right)^{1/2} \norm{y}_2,
\end{equation} 
where we abbreviate
\begin{equation}
M_N^{(-2)} =  \sum_{k \in N_R} \frac{1}{r_k^2}.
\end{equation} 
We know from \eqref{100} that $D_N$ is bounded $\P^{\otimes\N}$-{a.s.} Consequently, {there is a $C>0$ such that
\begin{equation*} 
\label{eq:nbrf}
\bnorm{\˘P}_{\mathrm{op}} \leq C \left(M_N^{(-2)}\right)^{1/2} \qquad \P^{\otimes\N}-a.s.
\end{equation*}  
}
Using the latter in combination with \eqref{estterm2} we have 
\begin{equation}
\label{estterm2a}
\begin{aligned}
\E_y[{\|S^N(t)n^N-\˘S^N(t)n^N\|^2_2]\,\,}
&{\, \leq\norm{\˘P}_{\mathrm{op}}^2\, \E_y[ \|PY^N(t)-Y^N(t)\|^2_2}]\\
&{\leq C M_N^{(-2)} g(t)} \\
&\leq CM_N^{(-2)} \left({\frac{\Sigma_N}{N M^-_N}} 
+ \bbbcls{g(0)-\frac{\Sigma_N}{N M^-_N}} \e^{-2N M^-_N t}\right)\\
& \leq C \left({\frac{\Sigma_NM_N^{(-2)}}{N M^-_N}} 
+ M_N^{(-2)}{g(0)}\,\e^{-2N M^-_N t}\right),
\end{aligned}
\end{equation}
{where we use  the definition of $Py$ and that $n^N_k<1$ for $k \in N_R$ in the second inequality, and \eqref{bound} in the third inequality.} By Assumption~\ref{ass}, the first term in the right-hand side of \eqref{estterm2a} converges to zero as $N\to\infty$ $\P^{\otimes\N}$-a.s. Again using \eqref{eq:nkykb} and $\e^{-x}\leq x^{-1}$, $x>0$, we see that the second term in the right-hand side of \eqref{estterm2a} converges to zero as $N\to\infty$ $\P^{\otimes\N}$-a.s.\ and uniformly in $t$ on compacts in $(0,\infty)$. 
\end{proof}

\begin{remark}
\label{rem:proj} 
{\rm Figure~\ref{fig1} illustrates why we need the various quantities in the right-hand side of \eqref{triangle}, which we discuss now.
\begin{compactenum}[label=(\alph*),wide]
\item 
While $Y^N(t)$ quickly moves close to its projection $PY^N(t) = S^N(t)n^N$ on the (blue colored) submanifold, and the distance between $Y^N(t)$ and $PY^N(t)$ can be controlled via the Lyapunov function argument, the motion of $PY^N(t)$ on the submanifold is difficult to approximate directly by a diffusion because the drift experienced by $Y^N(t)$ \emph{does not align} with the direction of the projection $P$. This means that $Y^N(t)$, and by extension also $PY(t)$, experiences a \emph{side-ways wind} across the submanifold, even when it is close to the submanifold. This wind is not seen by $PY^N(t)$ directly, which is therefore not obviously close to Markov. In contrast, for the transformed quantity $Z^N(t)$ with $Z^N_k(t) = Y^N_k(t)/r_k$, $k\in N_R$, its projection $\check{Z}^N(t)$ and drift \emph{do align}, making the motion of $\check{Z}^N(t)$ on the submanifold close to Markov, i.e., the drift term of $\check{Z}^N(t)$ vanishes (see \eqref{driftvanishes} below). 
\item 
Note that $\check{P}Y^N(t)$ is just a rescaled version of $\check{Z}^N(t)$, and the scaling is chosen so as to bring $\check{P}Y^N(t)$  back to the neighbourhood of $PY^N(t)$. Fortunately, the wind is not too strong, i.e., the direction of the main drift and the projection are close enough so long as $Y^N(t)$ is close to the submanifold. Hence, the distance between $\check{P}Y^N(t)$ and $PY^N(t)$ can be controlled by means of the distance between $Y^N(t)$ and $PY^N(t)$.
\item Figure~\ref{fig1} is drawn in correct proportions using the values $n^N = (0.37,0.63)$, $Y^N(t)=(0.17,0.61)$ and $r=(0.8,3.5)$. However, note that the actual drift experienced by the processes is much larger due to the additional factor $N$ appearing in their generators, which is not reflected in the figure.
\end{compactenum}}
\end{remark}


\subsection{Third term}
\label{sec:L3}

\begin{proof} 
Since
\begin{equation}
\bnorm{\˘S^N(t)n^N-\˘S^N(t)\P}_2 \leq \˘S^N(t) \bnorm{n^N-\P}_2,
\end{equation}
the result is a straightforward consequence of the strong law of large numbers for the multinomial distribution and the observation that $\˘S^N(t)\leq 1$. 
\end{proof}


\subsection{Fourth term}
\label{sec:L4}

The proof uses the following lemma.

\begin{lemma}[{\cite[Chapter~4, Corollary~8.4]{EK86}}] 
{Let $(K_N)_{N\in\N}$ be any increasing sequence of positive integers.} Put $E=[0,1]$ and $E_N = \{y \in E^{K_N}\colon\, y_l\leq n^N_l, \,l\in[K_N]\}$, both endowed with the respective Euclidean distances. Let $A\subset C_b(E)\times C_b(E)$ be linear and such that the closure of $A$ generates a strongly continuous contraction semigroup on the domain $\cD$ of $A$ (which is assumed to be separating). Let $X$ be a Markov process with generator $A$. For each $N \in \N$, let $Y_N$ be a progressive Markov process on $E_N$ corresponding to a measurable contraction semigroup with full generator ${\hat{A}}_N\subset C_b(E_N)\times C_b(E_N)$. Let $\eta_N\colon\,E_N\to E$ be measurable.\\[0.2cm] 
Suppose that, for each $(f,g)\in A$ and $T>0$, there exist $(f_N,g_N)\in{\hat{A}}_N$ such that
\begin{equation}
\label{101a}
\sup_{N \in \N} \sup_{0\leq t\leq T}\IE_N[\abs{f_N(Y_N(t))}] < \infty,
\qquad \sup_{N \in \N} \sup_{0\leq t\leq T}\IE_N[\abs{g_N(Y_N(t))}] < \infty,
\end{equation}
and
\begin{equation}
\label{101}
\begin{aligned}
&\lim_{N\to\infty}\IE_N\big[\babs{f_N(Y_N(t))-f(\eta_N(Y_N(t)))}\big] = 0,\\ 
&\lim_{N\to\infty}\IE_N\big[\babs{g_N(Y_N(t))-g(\eta_N(Y_N(t))}\big] = 0,
\end{aligned}
\end{equation}
where $\IE_N$ denotes expectation with respect to $Y_N(t)$. Then the finite-dimensional distributions of $\eta_N(Y_N)$ converge to those of $X$ as $N \to\infty$.
\end{lemma}

\begin{proof} 
All spaces are compact, and hence separable and complete. Note that, for each $t$, $\clc{X_N(t)}_{N\in\N}$ is relatively compact, since it is a subset of a compact set. The statement now follows from \cite[Chapter~4, Corollary~8.4]{EK86} after taking $M=\cD$ therein and using Eq.~(8.12) therein.
\end{proof}
 
\noindent 
We are now ready to show that the fourth term vanishes in expectation.

\begin{proof}
Let $\˘z$ be defined by
\begin{equation}
\label{zcheck}
\˘z_k(y) = n^N_k \sum_{l \in N_R} \frac{y_l}{r_l}.
\end{equation} 
From \eqref{LNdef} (or otherwise) it is not difficult to see that
\begin{equation}
\label{driftvanishes}
\begin{aligned}
&(L^N f)(\˘z(y))\\ 
& = -N \sum_{k \in N_R} \left(y_k - n^N_k \sum_{l \in N_R} y_l\right) \sum_{m \in N_R} n^N_m f_m(\˘z)\\ 
&\quad + \frac{1}{2} \sum_{k \in N_R} \frac{1}{r_k} 
\left(y_k-2y_k \sum_{l \in N_R} y_l + n^N_k \sum_{l \in N_R} y_l\right) \sum_{m,m'\in N_R} n^N_m n^N_{m'} f_{mm'}(\˘z)\\ 
&\quad + \bigo(N^{-1}\abs{f}_3) \\
&= \frac{1}{2} \left( \sum_{k \in N_R} \˘z_k - 2 \sum_{k \in N_R} \˘z_k \sum_{l \in N_R} y_l 
+  \sum_{k \in N_R} \frac{n^N_k}{r_k} \sum_{l \in N_R} y_l\right) \sum_{m,m' \in N_R} n^N_m n^N_{m'} f_{mm'}(\˘z)\\ 
&\quad + \bigo(N^{-1}\Sigma_N\abs{f}_3),  
\end{aligned}
\end{equation}
where we use that $\sum_{l \in N_R} y_l/r_l=\sum_{k} \˘z_k$. Abbreviate
\begin{equation}
\label{del}
\Delta^N_k(y) = y_k - n^N_k \sum_{l \in N_R} y_l.
\end{equation}
Note that
\begin{equation}
\sum_{k \in N_R} \frac{\Delta^N_k(y)}{r_k} = \sum_{k \in N_R} \frac{y_k}{n_k} - \sum_{k \in N_R} \frac{n^N_k}{r_k} \sum_{l \in N_R} y_l,
\end{equation}
and, hence,
\begin{equation}
\sum_{l \in N_R} y_l = D^N \sum_{k \in N_R} \˘z_k - D^N \sum_{k \in N_R} \frac{\Delta^N_k(y)}{r_k}.
\end{equation}
We can therefore write 
\begin{equation}
\begin{aligned}
(L^N f)(\˘z(y)) 
&= \frac{1}{2} \Bigg(\sum_{k \in N_R} \˘z_k - 2 \sum_{k \in N_R} \˘z_k 
\left(D^N \sum_{k \in N_R} \˘z_k - D^N \sum_{k \in N_R} \frac{\Delta^N_k(y)}{r_k}\right)\\
&\qquad\qquad +  (D^N)^{-1}\left( \sum_{k \in N_R} \˘z_k - D^N \sum_{k \in N_R} \frac{\Delta_k(y)}{r_k}\right)\Bigg)\\ 
&\qquad \times \sum_{k,l \in N_R} n^N_k n^N_l f_{kl}(\˘z) + \bigo(N^{-1}\Sigma_N\abs{f}_3)\\
&= \left(\sum_{k \in N_R} \˘z_k \left(1- D^N \sum_{k \in N_R} \˘z_k\right)
- \frac{1}{2} \left(1- 2D^N\sum_{k \in N_R} \˘z_k\right) \sum_{k \in N_R} \frac{\Delta^N_k(y)}{r_k}\right)\\
&\qquad \times \sum_{m,m' \in N_R} n^N_m n^N_{m'} f_{mm'}(\˘z) + \bigo(N^{-1}\Sigma_N\abs{f}_3).
\end{aligned}
\end{equation}
Hence, for $\˘s^N(y) = D^N \sum_{k \in N_R} \˘z_k$, we obtain
\begin{equation}
\begin{aligned}  
&(L^N f)(\˘s^N(y))\\ 
&= (D^N)^2 \left(\sum_{k \in N_R} \˘z_k \left(1- D^N \sum_{k \in N_R} \˘z_k\right)
- \frac{1}{2} \left(1- 2D^N \sum_{k \in N_R} \˘z_k\right) \sum_{k \in N_R} \frac{\Delta^N_k(y)}{r_k}\right) f''(\˘s^N)\\
&\quad + \bigo(N^{-1}\abs{f}_3)\\
&= \left(D^N\˘s^N(1-\˘s^N) - \frac12D^N(1-2\˘s^N) \sum_{k \in N_R} \frac{\Delta^N_k(y)}{r_k}\right) f''(\˘s^N) 
+ \bigo\bclr{N^{-1}\Sigma_N\abs{f}_3}.
\end{aligned}
\end{equation}
{Recall $L$ from \eqref{eq:gen}.} Now, let
\begin{equation}
 A = \bclc{(f,Lf)\colon\,f\in C^3([0,1])}\subset C_b([0,1])\times C_b([0,1]).
\end{equation}
The fact that the closure of $A$ generates a strongly continuous contraction semigroup (i.e., Feller semigroup) corresponding to the Fisher-Wright diffusion on $[0,1]$ is standard. Moreover, it is clear that {$C^3([0,1])$} is an algebra and separates points. The processes $Y_N$ are right-continuous and measurable with respect to their own filtration, hence progressive. Define the linear operator
\begin{equation}
A_N = \bclc{(f,L^Nf)\colon\,f\in C^3(E_N)}\subset C_b(E_N)\times C_b(E_N),
\end{equation}
the closure of which generates the process $Y_N$ (we will not need the flexibility of the full generator). 

Write $\eta_N(y) = \˘s(y) = D^N\sum y_l/r_l$. Let $f\in C^3([0,1])$ be arbitrary, and let $g = Lf$. Set $f_N(y) = f(\eta_N(y))$, which clearly is an element of $C^3(E_N)$, and hence an element of the domain of $A_N$. The first conditions of \eqref{101a} and \eqref{101} are therefore trivially satisfied. With $g_N = L^N f_N$,
\begin{equation}
\babs{g(\eta_N(y))-g_N(y)}\\
\leq \abs{D^N-D}\abs{f}_2 + \frac12 \abs{f}_2 D^N \sum_{k \in N_R} \frac{1}{r_k}
\abs{\Delta^N_k(y)} + \bigo\bclr{N^{-1}\Sigma_N\abs{f}_3}.  
\end{equation}
Consequently, from the above, taking expectations and using \eqref{bound}, we get
\begin{equation}
\label{ekin}
\begin{aligned}
\mel\IE_y\big[\babs{g(\eta_N(Y_N(t)))-g_N(Y_N(t))}\big]\\
&\leq \abs{D^N-D}\abs{f}_2 +  D^N \abs{f}_2
\bbbclr{M_N^{(-2)} \IE_y\left[\sum_{k \in N_R} \Delta^N_k(Y_N(t))^2\right]}^{1/2} 
+ \bigo\bclr{N^{-1}\Sigma_N\abs{f}_3 }\\ 
&\leq \abs{D^N-D}\abs{f}_2 \\
&\quad + D^N \abs{f}_2
\bbbclr{M_N^{(-2)} \bigg(\frac{\Sigma_N}{N M^-_N}
+ \bbbcls{g(0)-\frac{\Sigma_N}{N M^-_N}}\e^{-2N M^-_N t}\bigg)}^{1/2}
+ \bigo\bclr{N^{-1}\Sigma_N\abs{f}_3}.
\end{aligned}
\end{equation}
Now, the first term in the right-hand side of \eqref{ekin} converges to zero as $N\to\infty$ $\P^{\otimes\N}$-a.s.\ because of \eqref{100}. The second term can be treated in the same way as \eqref{estterm2a}. Finally, note that there exists a $C$ such that, for $N$ large enough, $M^-_N\leq C M_N^{(-2)}$, and hence
\begin{equation}
\frac{\Sigma_N}{N } = \frac{\Sigma_N M^-_N}{N M^-_N} \leq \frac{C\Sigma_NM_N^{(-2)}}{N M^-_N}.
\end{equation}
Therefore, the second part of \eqref{101} holds $\P^{\otimes\N}$-a.s. Finally, 
\begin{equation}
\abs{g_N(y)}\leq \abs{Lf(\eta_N(y))} + \abs{L^N (f\circ\eta_N)(y)-Lf(\eta_N(y))}.
\end{equation}
The first term is bounded since $f\in C^3([0,1])$ and the second term is bounded in expectation due to \eqref{ekin}, which implies the second part of \eqref{101a}. {This implies convergence of the finite-dimensional distributions on the real line.}

{From \cite[Theorem~2.1]{BM16} we have that, for real-valued stochastic processes, convergence of the finite-dimensional distributions on the real line implies weak convergence with respect to the topology induced by convergence of measure, which is equivalent to the Meyer-Zheng topology, hence no tightness is needed. We conclude that $\˘S^N$ converges to $S$ weakly with respect to the Meyer-Zheng topology. Since~$\Xi$ is a Polish space, we can use Skorokhod's representation theorem to couple $\˘S^N$ and $S$ such that $\˘S^N$ converges to $S$ almost surely. As noted after Corollary 2.2 of \cite{BM16}, convergence of uniformly bounded processes with respect to the convergence in measure topology, implies $L_2$ convergence, which concludes the proof.}

\end{proof}


\Addresses


\begin{thebibliography}{4}


\bibitem[BM16]{BM16}
V.\,I.~Bogachev and A.\,F.~Miftakhov (2016).
On the convergence of measures in the Wasserstein metric.
{\em Theory Stoch. Process.} \textbf{21}, 1--11.


\bibitem[DM78]{DM78}
C.~Dellacherie and P.-A. Meyer (1978).
\emph{Probabilities and {P}otential {A}: {G}eneral {T}heory}.
North Holland Publishing Company, New York.

\bibitem[Dur08]{DU08}
R.~Durrett (2008). {\em Probability Models for {DNA} Sequence Evolution}.
Probability and its Applications (New York). Springer, New York, second edition

\bibitem[EK86]{EK86}
S.\,N.~Ethier and T. G.~Kurtz (1986).
{\em Markov Processes --- Characterization and Convergence}.
Wiley Series in Probability and Mathematical Statistics: Probability and Mathematical Statistics. John Wiley \& Sons, Inc., New York.

\bibitem[GdHO22]{GdHO21}
A.~Greven, F.~den Hollander, and M.~Oomen (2022).
Spatial populations with seed-bank: well-posedness, duality and equilibrium.
{\em Electron. J. Probab.} \textbf{27}, article no.~18, 1--88.

\bibitem[Kur91]{Kurtz1991a}
T.\,G.~Kurtz (1991).
Random time changes and convergence in distribution under the Meyer-Zheng conditions.
{\em Ann. Probab.} \textbf{19}, 1010--1034.

\bibitem[Mei17]{M08}
J.\,D.~Meiss (2017).
{\em Differential Dynamical Systems}, volume~22 of {\em Mathematical Modeling and Computation}.
Society for Industrial and Applied Mathematics (SIAM), Philadelphia, PA, revised edition.

\bibitem[MZ84]{MZ84}
P.-A.~Meyer and W.\,A.~Zheng (1984).
Tightness criteria for laws of semimartingales.
{\em Ann. Inst. H. Poincar\'{e} Probab. Statist.} \textbf{20}, 353--372.

\bibitem[ZK01]{ZK01}
N.\,S.~Zakrevskaya and A.\,P.~Kovalevski\u{\i} (2001).
One-parameter probabilistic models of text statistics.
{\em Sib. Zh. Ind. Mat.} \textbf{4}, 142--153.


\end{thebibliography}
\end{document}